%% file: Rational_Curves_on_Fibered_Varieties.tex
\documentclass[a4paper,11pt]{amsart}
\usepackage{amscd,amsmath,amsthm,amssymb,marvosym}
\usepackage[english]{babel}
\usepackage[all]{xy}
\usepackage[colorlinks=true, linkcolor=blue, citecolor=blue]{hyperref}
\usepackage{etoolbox}

\usepackage{cite}
\begin{document}
	
	\newtheorem{theorem}{Theorem}[section]
	\newtheorem{lemma}[theorem]{Lemma}
	\newtheorem{proposition}[theorem]{Proposition}
	\newtheorem{corollary}[theorem]{Corollary}
	\newtheorem{question}[theorem]{Question}

	\theoremstyle{definition}
	\newtheorem{definition}[theorem]{Definition}

	\theoremstyle{remark}
	\newtheorem{example}[theorem]{Example}
	\newtheorem{remark}[theorem]{Remark}

	%operatori matematici
	\newcommand{\cod}{\operatorname{cod}}
	\newcommand{\dime}{\operatorname{dim}}
	\newcommand{\Hom}{\operatorname{Hom}}
	\newcommand{\Ker}{\operatorname{Ker}}
	\renewcommand{\ker}{\operatorname{Ker}}

	%simboli mathbb
	
	\newcommand{\K}{\mathbb{K}}
	\newcommand{\R}{\mathbb{R}}
	\newcommand{\N}{\mathbb{N}}
	\newcommand{\C}{\mathbb{C}}
	\newcommand{\Q}{\mathbb{Q}}
	\newcommand{\Z}{\mathbb{Z}}
	\newcommand{\Proj}{\mathbb{P}}
	\newcommand{\A}{\mathbb{A}}
	\newcommand{\F}{\mathbb{F}}
	\newcommand{\G}{\mathbb{G}}

	% simboli mathcal
	
	\newcommand{\Oh}{\mathcal{O}}
	\newcommand{\sA}{\mathcal{A}}
	\newcommand{\sB}{\mathcal{B}}
	\newcommand{\sC}{\mathcal{C}}
	\newcommand{\sE}{\mathcal{E}}
	\newcommand{\sF}{\mathcal{F}}
	\newcommand{\sG}{\mathcal{G}}
	\newcommand{\sH}{\mathcal{H}}
	\newcommand{\sI}{\mathcal{I}}
	\newcommand{\sK}{\mathcal{K}}
	\newcommand{\sM}{\mathcal{M}}
	\newcommand{\sN}{\mathcal{N}}
	\newcommand{\sP}{\mathcal{P}}
	\newcommand{\sR}{\mathcal{R}}
	\newcommand{\sT}{\mathcal{T}}
	\newcommand{\sU}{\mathcal{U}}
	\newcommand{\sV}{\mathcal{V}}
	\newcommand{\sY}{\mathcal{Y}}
	
	\newcommand{\NS}{NS}

	\newcommand\blfootnote[1]{%
		\begingroup
		\renewcommand\thefootnote{}\footnote{#1}%
		\addtocounter{footnote}{-1}%
		\endgroup
	}

%Dipartimento di Matematica e Fisica, Universita di Roma Tre, Largo San Leonardo Murialdo 1, `
%00146 Roma, Italy.

\address{Fabrizio Anella\\ Dipartimento di Matematica e Fisica\\ Universit\`{a} Roma 3\\ Largo San Leonardo Murialdo 1\\ 00146\\ Rome\\ Italy}
\email{fabrizio.anella2@uniroma3.it}
	
	\author{Fabrizio Anella}
	\date{March 2019}
	\title{Rational curves on fibered varieties}

	\begin{abstract}
		Let $X$ be a projective variety with log terminal singularities and vanishing augmented irregularity. In this paper we prove that if $X$ admits a relatively minimal genus one fibration then it does contain a subvariety of codimension one covered by rational curves contracted by the fibration. We then focus on the case of varieties with numerically trivial canonical bundle and we discuss several consequences of this result.
	\end{abstract}
\maketitle
\blfootnote{\textup{2010} \textit{Mathematics Subject Classification}. Primary: 14J32; Secondary: 14E30, 14D06.} 
%\subjclass[2010]{ Primary: 14J32; Secondary: 14E30, 14D06.}
\blfootnote{\emph{Key words and phrases}. Elliptic fiber space, genus one fibration, Calabi--Yau variety, fibration, Fisher--Grauert Theorem, rational curve.} 

	\section*{Introduction}
	\input{introduction}
	\subsubsection*{Acknowledgements} \input{acknowledgements}

	\section{Genus one fibrations} 
	\subsection{Definitions and notations}
	\input{notazioni}

	\subsection{Fischer--Grauert Theorem}\label{fisch}
	\input{fischergrauert}
	\subsection{Proof of Theorem \ref{due}}
	\input{preamboloprimo}
	\input{illuogosingolarecodimensione}
	\input{teoremasimilediverio}

	\input{esempio}
	\subsection{Trivial canonical bundle}
	\input{razionaliincanonicobanale}
	\input{conseguenzecanonicobanale}

\section{Fibration over curves}\label{sectiondue}
	\input{iniziosecondocapitolo}
	\subsection{Preliminar results}

\input{ampioeffettivo}
	\input{cosasullaretta}
	\subsection{Proof of Theorem \ref{newovercurves}}
	\input{dimostrazioneth}
	\input{conseguenzeultimoteorema}
	\bibliographystyle{alpha}
	 
\bibliography{C:/Users/fabrizio/Desktop/uni/personali/biblio}{}

\end{document}

%% file: introduction.tex
% !TeX spellcheck = en_GB

Finding rational curves in a projective variety $X$ is useful to understand the geometry of $X$ because these curves are strongly related to many invariants. Rational curves on Calabi--Yau varieties are particularly useful but the existence of such curves in full generality on these varieties is proven only in dimension two by Bogomolov--Mumford \cite{mori1983uniruledness}. %The existence of rational curves in Calabi-Yau manifolds is predicted also by phisics 
On K3 surfaces there are rational curves in any ample linear series. This leads to define Beauville--Voisin class as the zero-cycle class of a point on a rational curve \cite{beauville2001chow}. In higher dimension doing this is more difficult because it is hard to find an ample divisor $H\overset{i}{\rightarrow} X$ with $i_*(CH_0(H))=\Z$, and moreover we do not expect in general that there is a divisor with $CH_0(H)=\Z$, e.g. $H$ rational. Let us briefly give a couple of other motivations: a rational morphism to a manifold without rational curves is everywhere defined; finding a rational curve on a variety implies that the variety is not hyperbolic in the sense of Kobayashi. This list can be made much longer. 

The experience with minimal model program suggests that even if one is mainly interested in smooth varieties, the natural setting is to allow at least log terminal singularities. The aim of this paper is to extend the results proven in \cite{diverio2016rational} in a singular setting typical of the minimal model program. We use some techniques that lead us to prove some new results also in the smooth case.

%In the first part of this paper there are some definitions and notations. I explain the relation between different definitions of isotrivial families. In particular I extend Fischer--Grauert Theorem in the \'etale topology when the fibers are curves. I give the proof of the following result.
The core of this paper is the following result.
\begin{theorem}
	\label{due}
	Let $X$ be a normal projective variety of dimension $n$ with at most log terminal singularities and vanishing augmented irregularity, \textsl{i.e.} the irregularity of any quasi-\'etale cover of $X$ is zero. Suppose that there exists a surjective morphism $\phi$ from $X$ to a variety $B$ of dimension $n-1$. If there exists a cartier divisor $L$ on $B$ such that $\phi^* L \sim  K_X$, then there exists a subvariety of codimension one in $X$ that is covered by rational curves contracted by $\phi$.
\end{theorem}
An important consequence of this theorem is the case with numerically trivial canonical bundle. In this context we focus on Calabi--Yau varieties
	as in Definition \ref{calabiyau}.
\begin{theorem}\label{kbanale}
	Let $X$ be Calabi--Yau variety. Suppose that there exists a surjective morphism $\phi$ from $X$ to a variety $B$ of dimension $n-1$. Then there exists a subvariety of codimension one in $X$ that is covered by rational curves contracted by $\phi$.
	
\end{theorem}

In the case of varieties with trivial canonical bundle we study what happens in the Beauville--Bogomolov decomposition. In Section \ref{sectiondue} we study the case of a fibration onto a curve. In particular we prove the following result.
\begin{theorem}
	Let $X$ be a Calabi--Yau variety. Suppose there exists a nef $\mathbb{Q}$-divisor $D$ with numerical dimension one such that $c_2(X)\cdot D=0 $ in $N^3(X)$. Then $X$ does contain rational curves.
\end{theorem}
%When one wants to prove the existence of rational curves on varieties with trivial canonical bundle there is an useful tool: one can suppose by contradiction that the nef cone coincides with the pseudo-effective cone. So starting from a rational nef non ample divisor with numerical dimension one, one can find many rational nef but not ample divisors. 

%We will prove that in such case there are two transverse fibration in abelian varieties. 

%For variety with trivial canonical bundle is conjectured that the Kobayashi pseudodistance between any to point is zero. We can prove such conjecture in this context.
%\begin{theorem}
%	Let $X$ be a smooth variety with vanishing augmented irregularity and without rational curves. Suppose that there exists a fibration in abelian varieties to a curve. (questo teorema non si può vedere)
%\end{theorem}

%% file: acknowledgements.tex
The author would like to thank Andreas H\"{o}ring, Edoardo Sernesi and Fabio Bernasconi for important
comments and discussions. The author would also like to thank his advisor, Simone Diverio, for the continuous help provided during this work and for climbing!

%% file: notazioni.tex
In this paper every variety will be an irreducible projective variety over the complex number. The variety $X$ will be always normal and of dimension $n\geq 2$. If $X$ has a morphism to a variety $Y$, then $X_Z$ is the base change for $Z\rightarrow Y$. The notations and standard properties about singularities that are used in this article can be found for example in \cite{kollar2008birational}. For the reader's convenience we recall some definitions that will be used in this paper. 
%\begin{definition}
%	Let $f:Y\rightarrow X$ a morphism between normal variety. The \emph{singular locus} of $f$ is 
%\end{definition}
\begin{definition}
	A morphism $f : X \rightarrow Y$ between normal
	varieties is called \emph{quasi-\'etale} if $f$ is quasi-finite and \'etale in codimension one.
\end{definition} 
\begin{remark}\label{purity}
	A quasi-\'etale morphism to a smooth variety is globally \'etale by standard argument on purity of the branch locus.
\end{remark}

\begin{definition}
	The \emph{irregularity} of a variety $Y$ is the non negative integer $q(Y):=h^1(Y,\mathcal{O}_Y)$. The \emph{augmented irregularity} of $Y$ is the following, not necessarily finite, positive integer
	$$\tilde{q}(Y):=\sup \{q(Z)\ | \ Z\rightarrow Y  \text{ is a finite quasi-\'etale cover} \}.$$
\end{definition}
For any variety the inequality $\tilde{q}(Y)\geq q(Y)$ holds. In general the equality does not hold. Moreover it may happen that this supremum is not achieved; this is exactly the case when the augmented irregularity is infinite. This can happen also in dimension one as we see in the following example.
\begin{example} \label{esempiocurva}
	The behaviour of the augmented irregularity for smooth curves is easy to describe using Riemann--Hurwitz formula. The augmented irregularity of a genus zero curve is zero. Indeed $\mathbb{P}^1$ is simply connected and it is regular. Any finite \'etale cover of a genus one curve is again a genus one curve by Riemann--Hurwitz formula, so $\tilde{q}(C)=1$. A curve $C$ with $g(C)\geq 2$ has a cover of degree $d$ from a curve $C'$ of genus $g(C')=d\cdot (g(C)-1)+1$. Indeed its fundamental group has subgroups of index $d$ arbitrary large. This subgroup corresponds to an \'etale cover $\tilde{C}$ of degree $d$, whose genus is given by Riemann--Hurwitz formula and equals $g(\tilde{C})=d(g-1)+1$.  So we can find an \'etale cover of $C$ with arbitrary large irregularity. Hence $\tilde{q}(C)=\infty$.
\end{example}
%\begin{definition}
%	(numerical dimension) Let $Y$ be a normal variety. The numerical dimension of a non trivial nef class in $x\in N^1(Y)$ is the maximum integer $k$ such that $x^k\neq 0 $ as element in $N^k(X)$.
%\end{definition}
%\begin{definition}
%	(Second Chern class) For a normal variety $Y$ smooth in codimension two the second Chern class is defined as the pushforward of $c_2(\tilde{Y})$ for some resolution $\tilde{Y}$ of $Y$. 
%\end{definition}
%\begin{definition}
%	(Big and small subset) A big open subset of a variety $Y$ is an subset $U\subset Y$ contained in an open subset $V\subset Y$ such that $cod_Y(V^c)\geq 2$. A closed subset $W\subset Y$ is small if it is contained in a closed subset $Z\subset Y$ with $cod_Y(Z)\geq 2$.
%\end{definition}

\begin{definition}
	A \emph{fibration} is a morphism between normal varieties with connected fibers. An \emph{genus one fibration} is a fibration such that the general fiber is a smooth genus one curve. An \emph{elliptic fibration} is a genus one fibration with a fixed section.
\end{definition}
\begin{definition}
	Let $f:X\rightarrow Y$ a surjective projective morphism of normal variety. The \emph{singular values} of $f$ is the following subset of $Y$
	$$\operatorname{Sv}(f)=\{y\in Y |\ \operatorname{dim}(f^{-1}(y))>\operatorname{dim}(X)-\operatorname{dim}(Y) \vee f^{-1}(y)\text{ is singular} \}.$$
\end{definition}
\begin{remark}
	The singular values of $f$ is the image of the singular locus of $f$. For the interested reader the definition of singular locus of a morphism can be found at the following link: \texttt{\url{http://stacks.math.columbia.edu/tag/01V5}}. We do not give the definition of singular locus of a morphism because the definition is too technical and we just need the given characterization of the image of the singular locus.
\end{remark}
One can associate to any elliptic curve a complex number called its $j$-invariant. This association is modular, which means that an elliptic fibration $f:Y\rightarrow B$ comes with a rational map $j:B \dashrightarrow \mathbb{P}^1 $ called $j$-map that is at least defined over the smooth values of $f$.
For some standard facts about the $j$-map of an elliptic fibration the references can be found in \cite{kodaira1963compact} or \cite{hartshorne2009deformation}.
\begin{remark}
	The $j$-invariant is well-defined also for genus one curve, \textsl{i.e.} without fixing the origin. Indeed different choices of the origin does not change the $j$-invariant because the translation of the origin is an automorphism of elliptic curves. In particular, as for elliptic curves, two genus one curves are isomorphic if and only if they have the same $j$-invariant.
	
	A genus one fibration $X\xrightarrow{f} B$ gives a rational map that we call the $j$-map as in the case of elliptic fibrations. To show this consider an open subset $U$ of $B$ contained in the smooth values of $f$ and in the regular part of $B$. Let $\Sigma$ be a general element in a very ample linear series on $X$ restricted over $U$. The pullback $\Sigma \times_B X \rightarrow \Sigma$ is a smooth elliptic fibration. Up to shrink $U$ out of the ramification of $\Sigma $ over $U$, we can suppose that $\Sigma \rightarrow U$ is a finite \'etale cover. If necessary we can consider a further finite \'etale cover $\Sigma '$ of $\Sigma $ such that the composition $\Sigma '\rightarrow U $ is Galois. Since $X\times_U \Sigma '\rightarrow \Sigma' $ is an elliptic fibration it is well-defined the $j$-function $ \Sigma'\rightarrow \C$. The regular functions on $U$ are the regular functions on $\Sigma' $ that are invariant under the Galois group. All the fibers in an orbit of the Galois group are isomorphic, so the $j$-map on $\Sigma' $ descends to a regular function on $U$, that is a rational map on $B$.
\end{remark}

\begin{remark}\label{twoiso}
	Consider the following two different definitions of isotriviality for a flat family. One can ask that two general fibers are isomorphic, or that the smooth fibers are isomorphic. In the general setting the first definition is strictly more general than the second one. An example of this situation is given by a degeneration of an Hirzebruch surface $F_n$ into an $F_m$ with $m> n$, \cite[See Example 1.2.11(iii)]{sernesi2007deformations}. For elliptic fibrations these two definitions coincide. Indeed a smooth degeneration of an elliptic curve is again elliptic by Kodaira's table \cite{barth2015compact}. Since the $j$-invariant is constant on a dense subset of the base it is constant. We can conclude that every smooth fiber is a smooth elliptic curve with the same $j$-invariant, so the smooth fibers are isomorphic. Since a smooth projective morphism \'etale-locally admits a section, the same statement holds for a smooth genus one fibration.
\end{remark}

%% file: fischergrauert.tex
A well-known theorem proved by Fischer and Grauert \cite{fischer1965lokal} tells us that a proper holomorphic submersion with isomorphic fibers is locally a product in the complex topology. This means that given a proper holomorphic submersion $f:X\rightarrow B$ between complex manifolds such that for any $t,s\in B$ the fibers $X_t$ and $X_s $ are isomorphic, then for any $p\in B$ there exists a neighborhood $U_p\subset B$ open in complex topology such that the family $X_{U_p}\simeq X_p\times U_p$ splits in a product over the base. The same statement does not hold in the Zariski topology as we can see in the following example.
\begin{example}
	Let $f:X'\rightarrow X$ be any finite unramified (hence \'etale) morphism between varieties of degree $d>1$. For example $f$ can be a finite unramified morphism of degree $d$ from a smooth curve of genus $d(g-1)$ to a smooth curve of genus $g$. For any $p\in X$, the fiber over $p$ is a scheme given by $d$ distinct reduced points. In particular any two fibers are isomorphic. However for any $U\subseteq X$ open in the Zariski topology, the preimage $U':=f^{-1}(U)$ is a non-empty Zariski-open subset of $X'$. In particular since $U'$ is connected it is not isomorphic to the product between $d$ points and $U$ that has $d$ connected components.
\end{example}
For the general philosophy about the relation between complex topology and \'etale topology one can expect that the same statement of Fischer--Grauert Theorem holds for the \'etale topology. Since we were unable to find a neat reference on this subject, for the reader's convenience we prove some statements that will be useful for what follows. We hope to address a more careful analysis on this problem in a forthcoming paper.

%Some authors use this fact without any proof but I found no references. I hope that a paper on this argument will be published by the author within a short time. Here I just give a proof of a stronger results that works for curves of positive genus because I just need this fact.
\begin{proposition}\label{isocurve}
	Let $Y\rightarrow B$ a smooth proper morphism between normal quasi-projective varieties such that for any $t\in B$ the variety $Y_t$ is a smooth curve of genus $g\geq 1$. Suppose moreover that for any $s,t\in B$ the curves $Y_t$ and $Y_s$ are isomorphic. Then there exists a finite \'etale morphism $\tilde{B}\rightarrow B $ such that the pullback $Y_{\tilde{B}} \simeq Y_t\times\tilde{B}$ is a product.
\end{proposition}
\begin{proof}
	Fix a point $0\in B$. By GAGA's principle we can consider $B$ and $Y_0$ as complex manifolds, in this way we can study the monodromy around zero as follows. Fix an integer number $n$ greater than three and consider the action of the fundamental group of the base on the first cohomology group of the central fiber with coefficient in $\Z_n$ $$\phi:\pi_1(B,0)\rightarrow \operatorname{Aut} (H^1(Y_0, \Z_n)).$$
	Since $Y_0 $ is a complete curve of genus $g$ the group $H^1(Y_0, \Z_n)\simeq \Z _n^{2g}$ is finite. This implies that $\operatorname{Aut}(H^1(Y_0, \Z_n))$ is finite and hence $\ker(\phi)\unlhd \pi_1(B,0)$ is a normal subgroup of finite index of the fundamental group of the base. By the standard correspondence between subgroup of index $d$ of $\pi_1(B,0)$ and \'etale cover of $B$ of degree $d$, the subgroup $\ker(\phi)$ corresponds to a finite \'etale cover $\tilde{B}$ of $B$. Moreover the action of $\pi_1(\tilde{B},\tilde{0})$ is trivial on the first cohomology group with coefficients in $\Z_n$ of the pullback family $Y\times_B \tilde{B}$. This construction, called $J_n$-rigidification, is useful because for $n\geq 3$ there are no automorphisms of a curve with positive genus acting in a trivial way on $H^1(C, \Z_n)$. In particular there exists a fine moduli space with a universal family $\mathcal{U}_{g,n}\rightarrow \mathcal{M}_{g,n} $ (see for example \cite{beauville1996complex}). The classifying morphism $\tilde{B}\rightarrow \mathcal{M}_{g,n} $ is constant because the morphism  $B\rightarrow \mathcal{M}_{g} $ is constant (this morphism is constant since all the fibers are pairwise isomorphic). It follows that there is a pullback diagram as follows
		\[ \xymatrix{
		 Y\times_B \tilde{B}=\tilde{B}\times_{\mathcal{M}_{g,n} }\mathcal{U}_{g,n} \ar[d]\ar[r]  & \mathcal{U}_{g,n} \ar[d]  \\
		\tilde{B} \ar[r] & \mathcal{M}_{g,n} 
	}\]
and since the classifying morphism $\tilde{B}\rightarrow \mathcal{M}_{g,n} $ is constant, the variety $Y_{\tilde{B}}$ is isomorphic to the product $\tilde{B}\times Y_0$.
\end{proof}
In this paper we need the previous result only for genus one fibrations. Since the previous proof use many topological tools, we give another proof more algebraic in spirit of the following statement, that is essentially Proposition \ref{isocurve} for curves with genus one.
\begin{proposition}\label{isotriviale}
	Let $Y\rightarrow B$ a smooth projective morphism between normal varieties such that for any $t\in B$ the variety $Y_t$ is isomorphic to a fixed curve of genus $g=1$, \textsl{i.e.} a smooth isotrivial genus one fibration. Suppose moreover $B$ is smooth, then there exists a finite \'etale morphism $\tilde{B}\rightarrow B $ such that the pullback $Y_{\tilde{B}} \simeq Y_t\times\tilde{B}$ is a product.
\end{proposition}
To prove this proposition we need two results. 
\begin{lemma}\label{lemmahart}
	If $f:Y\rightarrow B$ is a smooth isotrivial elliptic fibration, then there exists a finite \'etale map $\tilde{B}\rightarrow B$ such that the pullback family $Y\times_{B}\tilde{B}$ is
	isomorphic to the trivial family.
\end{lemma}
This result is \cite[Corollary 26.5]{hartshorne2009deformation}. The big difference between Proposition \ref{isotriviale} and Lemma \ref{lemmahart} is that in the lemma the family of genus one curves has a section. So we have to combine this result with the following.
\begin{lemma}\label{lemma17}
	Let $f:Y \rightarrow B$ be a projective morphism between normal varieties. Assume that $B$ is smooth and $f$ is \'etale locally trivial and the generic fiber $F$ has numerically trivial canonical bundle. Then there is a finite \'etale cover $B'\rightarrow B$ such that the pull back $Y_{B'}\simeq F\times B'$ is globally trivial.
\end{lemma}
This lemma is stated and proved in \cite[Lemma 17]{kollar2009quotients}. Finally we can give an algebraic proof of Proposition \ref{isotriviale}.
\begin{proof}[Proof of \ref{isotriviale}]
	We have to prove that $f$ is \'etale locally trivial, \textsl{i.e.} for any $p\in B$ there exists an \'etale neighborhood $U$ of $p$ such that $Y_U\simeq U\times Y_p$. Choose a point $p\in B$. The morphism is smooth and projective so locally around $p$ there exists a multi-section $\Sigma$ of $f$ that is \'etale at $p$. Indeed the local structure of smooth morphism can be described in the following way: for any point $y\in Y$ and $t=f(y)$ there exist open neighborhood $V_t$ and $U_y$ with $U_y\subset f^{-1}(V_t)$ such that $f|_{U_t}$ factorizes as an \'etale morphism $g:U_y\rightarrow \mathbb{A}^d_{V_t}$ followed by the canonical projection $\mathbb{A}_{V_t}^d\rightarrow V_t$. Consider a section $s$ of $\mathbb{A}_{V_t}^d\rightarrow V_t$ and the associated fiber product $U\times_{\mathbb{A}^d_V}s(V_t)$. The image of this fiber product in $U_y$ is the desired \'etale multi-section.
	Shrinking $\Sigma $ we can suppose that the fiber product $Y_{\Sigma}\rightarrow \Sigma $ is a family of smooth elliptic curves and the fibers are pairwise isomorphic, so by Lemma \ref{lemmahart} $Y_{\Sigma} \simeq \Sigma \times Y_p$. This proves that $f$ is \'etale locally trivial. We can apply Lemma \ref{lemma17} and the proof is completed.
\end{proof}

%% file: preamboloprimo.tex
Before start proving Theorem \ref{due} we need a lemma that is essentially stated in \cite{diverio2016rational}, but with some more extra care in the case the fibration has no sections.

%% file: illuogosingolarecodimensione.tex
\begin{lemma}
	\label{isocod}
	Let $X\overset{\pi}{\rightarrow} B$ be a genus one fibration. If the subvariety of singular values $Z:=\operatorname{Sv}(\pi)$ has codimension at least two then the family $\pi $ is isotrivial.
\end{lemma}
\begin{proof} 
	Since $B$ is normal it is smooth in codimension one and also the subvariety $Z\cup B_{\text{sing}}$ has codimension at least two. We denote $B_0:=Z^c\cap B_{\text{reg}}$. The $j$-map $B\dashrightarrow \mathbb{P}^1$ is well-defined on $B_0$. Moreover the image of $B_0$ under this map is contained in $\mathbb{A}^1_\C $. Since $B$ is normal and $(B_0)^c$ has codimension at least two, this map extends to a holomorphic function $j:B\rightarrow \C$. This function must be constant because $B$ is projective and this means that $\pi $ is isotrivial.
\end{proof}

%\begin{remark}\label{iso2}
%	Since $B$ is normal the condition $cod_B(Z)\geq 2$ is equivalent to $cod_B(\pi(sing(\pi))\geq 2$. Since the locus where the preimage of a point has codimension greater than two has codimension at least two in $B$ this is also equivalent to $Z_0:=cod_B(\pi(sing(\pi))\cap \pi(Exc(\pi)))\geq 2$.
%\end{remark}

%% file: teoremasimilediverio.tex
For the reader's convenience we state again the theorem that we are going to prove.

\begin{theorem}\label{duee}
	Let $X$ be a projective variety with at most log terminal singularities and $\tilde{q}(X)=0$. Suppose that there exists a surjective morphism $\phi:X\rightarrow B$ to a variety of dimension $n-1$. If there exists a Cartier divisor $L$ on $B$ such that $\phi^* L \sim  K_X$, then there exists a subvariety of codimension one in $X$ that is covered by rational curves contracted by $\phi$.
\end{theorem}

\begin{proof}
	The proof is divided in several steps some of which might be already known to the experts. In particular Step 1, Step 3 and Step 5 adapt arguments used in \cite{diverio2016rational}.

	$Step\  1:$ the morphism $\phi:X\rightarrow B$ is a genus one fibration. We can suppose, by taking the normalization of $B$ and passing to Stein factorization, that the morphism $\phi $ has connected fibers and the base $B$ is normal. For dimensional reasons the generic fiber is a curve. Since $X$ is normal $X_{\text{sing}}\subset X$ has codimension at least two, so $\phi(X_{\text{sing}})\subset B$ has positive codimension. The restriction on the regular part of $X$ is a morphism from a smooth variety, so there is a non-empty open subset $U\subset B$ where the morphism $\phi:X\cap \phi^{-1}(U) \rightarrow U$ is smooth \cite[III.10.7]{hartshorne2013algebraic}. Let $Z\subset B$ be the union of the singular locus of $B$ and the singular values of $\phi$, \textsl{i.e.} $Z:=B_{\text{sing}}\cup \operatorname{Sv}(\phi)$. Now $B_0:=Z^c$ and $B_0\cap \phi(X_{\text{sing}})^c$ are non-empty open sets and the morphism  $\phi_0:X_0:=\phi^{-1}(B_0)\rightarrow B_0$ is a smooth proper surjective morphism. Taking the determinant of the relative cotangent bundle sequence restricted to a fiber that is in the regular part of $X$, we get the isomorphism $K_E\sim K_{X_0}|_E$. It follows that $ K_E \sim  K_{X_0}|_E \sim \phi^* L |_E\sim \mathcal{O}_E$. A smooth curve with trivial canonical bundle is a genus one curve and a smooth degeneration of a genus one curve has again genus one \cite[See Section V.7]{barth2015compact}, so every fiber of $\phi_0:X_0\rightarrow B_0$ is a smooth genus one curve.

	$Step \ 2$: we reduce to the case where the subvariety $Z$ has codimension one in $B$. Suppose every irreducible component of $Z$ has codimension at least two. By Lemma \ref{isocod} the family $\phi$ is isotrivial, so by Proposition \ref{isocurve} or \ref{isotriviale} there exists a variety $C_0 $ and a finite \'{e}tale cover $C_0\overset{\tau}{\rightarrow} B_0 $ such that the pullback $C_0\times_{B_0}X_0$ is globally trivial, \textsl{i.e.} $C_0\times_{B_0}X_0\overset{\psi}{\cong} C_0\times E$. The morphism induced by the following diagram 

	\[ \xymatrix{
	C_0\times E \ar[r]^{\psi^{-1}}  & X_0 \times_{B_0} C_0 \ar[d] \ar[r]^{\tau '}& X_0 \ar[d]  \\
	& C_0 \ar[r]^\tau & B_0 
}\]
given by the composition $\alpha_0:=\psi^{-1}\circ\tau ':C_0\times E \rightarrow X_0$ 
	is finite \'{e}tale because $\tau$ is the pullback of a finite \'{e}tale morphism. In particular the composition of the morphisms $C_0\times E \overset{\alpha_0}{\rightarrow}X_0\overset{i}{\rightarrow} X$ is quasi-finite and \'etale. By Zariski's Main Theorem \cite{grothendieck1967elements} a quasi-finite morphism is always the composition of an open immersion and a finite morphism, so there is a commutative diagram
	\[ \xymatrix{
		C_0\times E \ar[r]^{\alpha_0} \ar[d]^{i'}  & X_0 \ar[d]^i  \\
		 Y \ar[r]^\alpha & X 
	}\]
	where $i'$ is an open immersion and $\alpha$ is a finite morphism. The exceptional locus of $\phi $ is by definition $\operatorname{Exc}(\phi) = \{x \in X \ | \  \dim (\phi^{-1}(\phi(x))) > 1 \}$. The variety $X_0^c=X_Z$ can be splitted as a disjoint union $$X_Z=(X_Z\cap \operatorname{Exc}(\phi))\sqcup (X_Z\cap \operatorname{Exc}(\phi)^c) .$$
	Since the subvariety $ X_Z\cap \operatorname{Exc}(\phi)^c$ has dimension at most $\dime(Z)+1$ and we are assuming that $\cod _B(Z)\geq 2 $, the dimension of $X_Z$ is bounded by $\dime(X_Z\cap \operatorname{Exc}(\phi)^c)\leq n-2$. Since $K_X\sim \phi^*(L)$ the anticanonical bundle $-K_X$ is $\phi$-nef, hence by \cite[Theorem 2]{kawamata1991length} the $\phi$-exceptional locus is covered by rational curves contracted by $\phi$ (in the cited reference Kawamata didn't say explicitly that the rational curves are contracted by $\phi$, however this is clear from his proof).  This implies that if the exceptional locus of $\phi$ has codimension one in $X$, it is a uniruled subvariety of codimension one of $X$. This allows us to assume that $\cod_X(X_Z)\geq 2$.
	
	Since $\alpha $ is finite, also $i'(C_0\times E )^c $ has codimension at least two in $Y$. In particular since $\alpha$ is \'{e}tale in $i'(C_0\times E)\subset Y$, this argument proves that $\alpha$ is a finite quasi-\'{e}tale cover of $X$, so by hypothesis $H^1(Y, \mathcal{O}_Y)=0$. By \cite[Proposition 5.20]{kollar2008birational} $Y$ has log terminal singularities. As proved in \cite[Proposition 6.9]{greb2011singular} there is an isomorphism $\overline{H^0(Y, \Omega^{[1]}_Y)}\simeq H^1(Y, \mathcal{O}_Y)$. 
	By definition $\Omega^{[1]}_Y$ is a reflexive sheaf, so it is isomorphic to the sheaf of one forms on the regular part. The variety $C_0 $ is smooth because it is a finite \'{e}tale cover of $B_0$, so $\Omega^{[1]}_Y=i'_*\Omega^1_{C_0\times E}$. Since 
	$$0=\overline{H^1(Y, \mathcal{O}_Y)}\simeq H^0(Y, \Omega^{[1]}_Y)\simeq H^0(C_0\times E, \Omega^1_{C_0\times E})=$$
	$$=H^0(C_0,\Omega^1_{C_0})\oplus H^0(E, \Omega^1_E) \neq 0$$
	we reach a contradiction, so if there are no uniruled divisors on $X$ then $Z$ has codimension one in $B$.

	$Step\ 3$: restriction to a fibration onto a curve with some singular values. Let $H$ be a very ample divisor on $B$ such that $(n-2)H+L$ is globally generated. The pullback $\phi^*H$ is a globally generated Cartier divisor. Moreover there is an isomorphism 
	$$H^0(X,\phi^*H)\simeq H^0(B,\phi_*(\phi^*H))\simeq H^0(B,H)$$
	because $\phi $ has connected fibers. This isomorphism implies that general elements in $|H|$ are general also in $|\phi^*(H)|$. So we can choose $n-2$ general divisors $D_1, \ldots , D_{n-2} \in |H|$ such that $C:=D_1\cap \ldots \cap D_{n-2}$ is a smooth irreducible curve in $B_{\text{reg}}$ not contained in Z and $S:=\phi^{-1}(D_1)\cap \ldots \cap \phi^{-1}(D_{n-2})$ is a normal surface. We call again $\phi$ the morphism $\phi|_{S}$. Since $Z$ has codimension one, it must intersect $C$. Indeed $Z\cdot C=Z\cdot D_1\cdot \ldots \cdot D_{n-2}=Z\cdot H^{n-2}>0$ because $H$ is ample in $B$. This means that $\phi$ must have some singular fiber. To prove the existence of a uniruled divisor in $X$ it is sufficient to find a rational curve in the general $S$.
	
	$Step\ 4$: the case where $\phi^{-1}(p_i)\cap  \text{sing(S)}\neq \emptyset$. Let $\overline{S}$ be a minimal resolution of $S$
	\[ \xymatrix{
		\overline{S} \ar[rd]^{\beta} \ar[d]^\nu \\
		S \ar[r]^\phi & C. 
	}\]
	We can assume $\beta$ is relatively minimal. Indeed if there is a $(-1)$-curve on $\overline{S}$ contracted by $\beta$, the image of such a curve is again a rational curve in $S$ because it cannot be contracted to a point by minimality of the resolution. If there are $(-1)$-curves in the general surface $S$ constructed above, then the union of such rational curves cover a divisor of $X$.
	Let $p_1,\ldots, p_k$ be the points of $C\cap Z$. The singular curves $\phi^{-1}(p_i)\subset S$ are exactly $\phi^{-1}(p_i)=\nu(\beta^{-1}(p_i))$.
	Since $\beta$ is a minimal genus one fibration, by Kodaira's table \cite[Section V.7]{barth2015compact} a fiber of $\beta$ can be a smooth genus one curve, a sum of (possibly non reduced) rational curves or a non reduced genus one curve. 
	If $\phi^{-1}(p_i)$ contains some singular point of $S$, then $\beta^{-1}(p_i)=\nu^{-1}(\phi^{-1}(p_i) )$ contains an exceptional divisor of $\nu$, in particular $\beta^{-1}(p_i)$ must be a sum of rational curves. Since not every rational curve of $\beta^{-1}(p_i)$ can be contracted in $S$, the curve $\phi^{-1}(p_i)=\nu(\beta^{-1}(p_i))$ must be a sum of rational curves in $S$.
	
	$Step\ 5$: the case where $\phi^{-1}(p_i)\subset S_{\text{reg}}$. The curve $\phi^{-1}(p_i)$ is the central fiber of a family $S_0\overset{\phi}{\rightarrow} \Delta$ of genus one curves. Since $\phi^{-1}(p_i)$ is not smooth, by Kodaira's table it is a rational curve or a non reduced irreducible genus one curve. We need to exclude the last possibility.

	%Taking as before $D_1,\dots, D_{n-2}$ general divisor passing through $p$ in a very ample linear series on $B$ I can consider the map restricted on a surface. There is an open subset $U\subset C$ such that $\phi^{-1}(p_i)$ is a curve in $S_0:=\phi^{-1}(U)\subset S_{reg}$. 
	By adjunction formula the canonical bundle of $S_{\text{reg}}$ is base point free. Indeed 
	$$ K_{S_{\text{reg}}}\sim ( K_X+(n-2)\phi^*H )|_{S_{\text{reg}}}\sim \phi^*(L+(n-2)H)|_{S_{\text{reg}}} $$
	the canonical bundle is the restriction of the pullback of a base point free divisor.

	By \cite[V.12.3]{barth2015compact} the canonical bundle of $S_{\text{reg}}$ can be computed using the formula
	$$K_{S_{\text{reg}}}\sim \phi^*D +\sum (m_i-1)F_i $$
	for some divisor $D$ on the base and the sum runs over all the multiple fiber $F_i$ with multiplicity $m_i$. The restriction of the canonical bundle of $S_{\text{reg}}$ to $F_i$ is base point free because $K_{S_{\text{reg}}}$ is base point free. By the above formula for any $i$ the canonical bundle restricted to $F_i $ is
	$$K_{S_{\text{reg}}}|_{F_i}\sim (\phi^*D +\sum (m_i-1)F_i)|_{F_i}=\mathcal{O}_{F_i}((m_i-1)F_i).$$
	that in particular has some sections because it is the restriction of a base point free line bundle.
	The line bundle $\mathcal{O}_F{_{i}}(F_i)$ is torsion of order $m_i$ by \cite[Lemma III.8.3]{barth2015compact}. A non-trivial torsion line bundle has no sections, so for any $i$ the line bundle $\mathcal{O}_{F_i}((m_i-1)F_i)$ must be trivial, hence the multiplicity of the fiber $m_i$ is one for any $i$ and this is a contradiction. 
\end{proof}

	Theorem \ref{due} is inspired by \cite{diverio2016rational} where they proved a similar result in the case $X$ is a smooth projective manifold with finite fundamental group.
	\begin{remark}\label{fundamentalgroup}		
		For a smooth projective variety $Y$ with finite fundamental group the augmented irregularity is trivial. Indeed a quasi-\'etale cover is an \'etale cover for purity of branch locus. The fundamental group of an \'etale cover $\tilde{Y}$ of $Y$ is a subgroup of the fundamental group of $Y$, so it is finite. The first Betti number of a variety with finite fundamental group is zero, so by Hodge theory also $H^1(\tilde{Y}, \mathcal{O}_{\tilde{Y}})=0$, and hence $\tilde{q}(Y)=0$.
	\end{remark}
	This remark implies that Theorem \ref{duee} is stronger than \cite[Theorem 1.1]{diverio2016rational} also for smooth varieties. %Moreover the condition $\tilde{q}(X)=0 $ is strictly stronger than the condition $q(X)=0$. Indeed there exists a smooth projective regular surface $S$ with a finite \'etale Galois cover $\tilde{S}\rightarrow S$ of an irregular variety. In particular $\tilde{q}(S)\geq q(\tilde{S})>0=q(S)$. This example can be found in \cite{lopes2012geography}. 
	%Their arguments prove the existence of a $n-1$ dimensional uniruled subvariety and not only a rational curve as stated in their paper.
	An interesting application of Theorem \ref{duee} is the following corollary, already observed in their context in \cite{diverio2016rational}.
	\begin{corollary}
		A variety $X$ with at most klt singularities, $\tilde{q}(X)=0$, $\kappa(X)=n-1$ and whose canonical bundle is nef of exponent one does contain a uniruled divisor.
	\end{corollary}
	We conclude this section with an example where one can apply Theorem \ref{duee}.

%% file: esempio.tex
\begin{example}
	Fix two integer numbers $r\geq 1$ and $d\geq 2$. Consider a smooth hypersurface $X_{3,r}\subset \mathbb{P}^2\times \mathbb{P}^d$ given by the zero locus of a bihomogeneous polynomial of bedegree $(3,r)$. Consider the natural projection $\pi:X_{3,r}\rightarrow \mathbb{P}^d$. The augmented irregularity of $X_{3,r}$ is zero because it is simply connected by Lefschetz hyperplane theorem. By Grothendieck--Lefschetz Theorem the Picard group of $X_{3,r}$ is isomorphic to $\operatorname{Pic}(\mathbb{P}^2\times \mathbb{P}^d)$ and by adjunction formula the canonical bundle is $K_{X_{3,r}}\sim \mathcal{O}_{X_{3,r}}(0, r-d-1)$. In particular $K_{X_{3,r}}\sim \pi^* \mathcal{O}_{\mathbb{P}^d}(r-d-1)$. So we can apply Theorem \ref{duee}: it follows that this kind of family of genus one curves can't be everywhere smooth but it degenerates over a divisor of the base in (singular) rational curves.
\end{example}

%% file: razionaliincanonicobanale.tex
The following is a possible definition of Calabi--Yau variety that we will use in this paper.
\begin{definition}\label{calabiyau}
	A \emph{Calabi--Yau variety} $X$ is a projective variety with at most log terminal singularities, $K_X\equiv 0$ and $\tilde{q}(X)=0$.
\end{definition}
 For Calabi--Yau varieties we can improve Theorem \ref{duee} proving the following result.
\begin{theorem}\label{canonic}
	Let $X$ be Calabi--Yau variety. Suppose that there exists a morphism $\phi:X\rightarrow B$ whose general fiber is a curve. Then there exists a uniruled subvariety of codimension one in $X$ that is covered by rational curve contracted by $\phi$.
\end{theorem}

\begin{proof}
	By global index one theorem \cite[Proposition 2.18]{greb2017klt} there is a variety $X'$ with canonical singularities and a finite quasi-\'{e}tale morphism 
	$ \alpha: X' \rightarrow X $ such that $K_{X'}\sim 0 $. A finite quasi-\'{e}tale cover $Y\rightarrow X'$ is also (after composition with $\alpha $) a finite quasi-\'{e}tale cover of $X$. This proves that $\tilde{q}(X')\leq \tilde{q}(X)$ and so $\tilde{q}(X')=0$. If there is a subvariety $V\subset X'$ of dimension $n-1$ that is covered by rational curves, then also the variety $\alpha(V)\subset X$ is covered by rational curves. Since the canonical bundle of $X'$ is linearly equivalent to the trivial line bundle it is automatically the pullback of the trivial line bundle. The hypotheses of Theorem \ref{duee} are satisfied, so the theorem is proved.
\end{proof}

%% file: conseguenzecanonicobanale.tex
To preserve the dichotomy given by Beauville--Bogomolov decomposition between irreducible symplectic varieties and Calabi--Yau varieties in the singular setting, a useful definition is given for example in \cite{greb2017klt}, \cite{druel2018decomposition}, \cite{horing2018algebraic} and related papers. In particular, in \cite{horing2018algebraic}, they prove that there exists a version of the Beauville--Bogomolov decomposition for varieties with canonical singularities and smooth in codimension two. In these definitions of Calabi--Yau varieties and irreducible symplectic varieties there are some conditions on the reflexive exterior algebra of forms, that in particular imply that such varieties must have vanishing augmented irregularity. 

Being a Calabi--Yau variety as in Definition \ref{calabiyau} means that in the Beauville--Bogomolov decomposition \cite[Theorem 1.5]{horing2018algebraic} the abelian factor is trivial. Without the assumption on the smoothness in codimension two an analogous statement is \cite[Theorem B]{greb2017klt}. 
In particular Theorem \ref{canonic} can be applied to any product of Calabi--Yau's and irreducible symplectic varieties with such a definition. 
But let us be more precise. 

Let $X$ be a variety with at most log terminal singularities. Suppose moreover $K_X\equiv 0$ and that there is a surjective morphism $\phi:X\rightarrow B$ to a variety of dimension $n-1$. By \cite[Theorem B]{greb2017klt} there is a quasi-\'etale map $f:A\times Y\rightarrow X$ with $A$ an abelian variety of dimension $\tilde{q}(X)$, and $\tilde{q}(Y)=0$. Passing through the Stein factorization of $\phi \circ f$ we get a genus one fibration $\alpha:A\times Y\rightarrow \tilde{B}$. If the restriction of $\alpha$ to $\{ t \}\times Y$ for generic $t$ is a family of curves, then we can apply Theorem \ref{canonic} and find a uniruled divisor on $\{ t \}\times Y$. This implies that there is also a uniruled divisor on $A\times Y$ and hence its image under $f$ is again a uniruled subvariety of codimension one in $X$. 
\begin{remark}
	We can't expect that we can always apply Theorem \ref{canonic} to the restriction of the fibration to $\{ t \}\times \tilde{X}$ because it may happen that $\alpha$ is a projection, \textsl{i.e.} $X=E\times Y\rightarrow Y$ for some genus one curve $E$.
\end{remark}
Theorem \ref{canonic} is a generalization of \cite[Corollary 1.2]{diverio2016rational}. Also for smooth varieties, Theorem \ref{canonic} seems more general than their result because in \cite{diverio2016rational} a Calabi--Yau variety must have finite fundamental group. Such finiteness condition is a priori stronger than the vanishing of the augmented irregularity (see Remark \ref{fundamentalgroup}). However one can see as consequences of Beauville--Bogomolov decomposition for smooth varieties that this two conditions are equivalent: a smooth projective variety with numerically trivial canonical bundle and vanishing augmented irregularity has finite fundamental group. It is conjectured that the same implication holds also in the singular case, at least for varieties with mild singularities.

%% file: iniziosecondocapitolo.tex
In this section we study the dual case of a genus one fibration: the case of a surjective morphism $\pi:X\rightarrow C$ to a curve. Passing through the Stein factorization we can assume $\pi$ has connected fibers and since $X$ is normal we can assume that $C$ is smooth. So it is sufficient to study the geometry of a morphism with connected fibers onto a smooth curve. %The idea is to find an elliptic fibration starting from a fibration onto a curve, 
A fiber of a morphism onto a curve is a semiample divisor with numerical dimension one. So it is a priori more general to work only with a nef divisor with numerical dimension one than with a fibration onto a curve. 

For the reader's convenience we recall the definition of numerical dimension of a nef divisor.
\begin{definition}\label{calabi}
	Let $Y$ be a normal variety. The \emph{numerical dimension} of a nef class $x\in N^1(Y)$ is the maximum integer $k$ such that $x^k\neq 0 $ as element in $N^k(X)$.
\end{definition}
In this section we prove the following statement.
\begin{theorem}\label{newovercurves}
	Let $X$ be a Calabi--Yau variety. Suppose there exists a nef $\mathbb{Q}$-divisor $D$ with numerical dimension one such that $c_2(X)\cdot D=0 $ in $N^3(X)$. Then $X$ does contain rational curves.
\end{theorem}
\begin{remark}
Theorem \ref{newovercurves} is a generalization of \cite[Theorem 1.6]{diverio2016rational} also for smooth varieties. Indeed for smooth varieties with trivial canonical bundle with a fibration onto a curve with general fiber an abelian variety $F$, the class of $F$ in $N^1(X)$ has numerical dimension one and intersect in zero the second Chern class of $X$, \textsl{i.e.} $F\cdot c_2(X)=0$ \cite[Section 3]{diverio2016rational}. Moreover a divisor with numerical dimension one which intersects in zero the second Chern class of $X$ is just conjecturally semiample.
\end{remark}
The geometric meaning of Theorem \ref{newovercurves} is clear if the divisor is also semiample. In this case the Itaka fibration associated to $D$ is a fibration onto a curve. A general fiber $F$ of such a morphism intersects trivially $c_2(X)$, \emph{i.e.} $c_2(F)=F\cdot c_2(X)=0$. If $F$ is contained in the regular part of $X$, then by adjunction formula $F$ has automatically trivial canonical bundle. This in particular implies that there is an abelian variety with a finite quasi-\'etale cover to $F$.

\subsection{Chern classes for singular varieties}
Let us spend some words about the Chern classes for singular varieties. The Todd class and the Chern classes) of an arbitrary algebraic scheme are defined in \cite[Section 18.3]{Ful84}. 

\begin{remark}
	Let $\pi:\tilde{X}\rightarrow X $ be a proper birational morphism that is an isomorphism outside $Z\subset X$. Then $$\operatorname{Td}(X)=\pi_*\operatorname{Td}(\tilde{X})+\alpha \in A_*(X)_{\mathbb{Q}}$$
	where $\alpha$ is a class supported in $Z$. In particular this tells us that if $X$ is a variety smooth in codimension two, the definition $c_2(X):=\pi_* c_2(\tilde{X})$ for some resolution $\pi:\tilde{X}\rightarrow X$ agrees with the definition in \cite{Ful84}. We want to prove that these two definitions agree also for varieties with rational singularities.
\end{remark}
\begin{remark}
	Using the definitions in \cite{Ful84} the Hirzebruch-Riemann-Roch Theorem \cite[Corollary 18.3.1]{Ful84} holds for any complete scheme. Let $X$ be a projective variety with rational singularities, \textsl{e.g.} with at most dlt singularities, and $\pi:\tilde{X}\rightarrow X$ a resolution of singularities. By definition of rational singularities for any line bundle $L$ on $X$ we have $\chi(X,L))=\chi(\tilde{X},\pi^* L)$. Applying Hirzebruch-Riemann-Roch to $X$ and to $\tilde{X}$ we get $$\int_X \operatorname{ch}(L)\cdot \operatorname{Td}(X)=\int_{\tilde{X}} \operatorname{ch}(\pi^*L)\cdot \operatorname{Td}(\tilde{X})=\int_X \operatorname{ch}(L)\cdot \pi_* \operatorname{Td}(\tilde{X}) $$
	where the last equality follows from the projection formula \cite[Proposition 2.5 (c)]{Ful84}. Since this equality holds for any line bundle $L$, this tells us that $c_2(X)= \pi_* c_2(\tilde{X})$ as elements in $N^2(X)$. In particular for a variety $Y$ with at most klt singularities we can take as definition $c_2(Y)=\pi_* c_2(\tilde{Y})$ for some resolution of the singularities of $Y$.
\end{remark}
The pseudo-effectiveness of the second Chern class proved by Miyaoka holds also in our setting.
\begin{lemma}\label{miyaoka}
	Let $X$ be a normal projective variety with canonical singularities and $K_X\equiv 0$. For any $D_1,\dots , D_{n-2}$ nef divisors on $X$ we have $c_2(X)\cdot D_1\cdots D_{n-2}\geq 0$.
\end{lemma}
\begin{proof}
	Let $\nu :\tilde{X} \rightarrow X $ be a terminalization of $X$. The canonical bundle of $\tilde{X}$ is still numerically trivial and $\tilde{X}$ is smooth in codimension two. The divisors $\nu^* D_1,\dots ,\nu^* D_{n-2}$ are nef, so applying \cite[Theorem 6.6]{miyaoka1987chern} and the projection formula we get $c_2(\tilde{X})\cdot D_1 \cdots D_{n-2}\geq 0$. The conclusion follows applying another time the projection formula to $\nu$.
\end{proof}
In our setting to prove that the second Chern class of $X$ is trivial it is sufficient to show that $c_2(X)\cdot H^{n-2}=0$ for some ample divisor $H$.
\begin{lemma}\label{intersecitonchern}
	Let $X$ be a normal projective variety with canonical singularities and $K_X\equiv 0$. Then $c_2(X)=0$ in $N^2(X)$ if and only if there exist $H_1,\dots , H_{n-2}$ ample line bundles on $X$ such that $c_2(X)\cdot H_1\cdots H_{n-2}=0$. In particular if $c_2(X)\neq 0$ in $N^2(X)$ then for any ample line bundle $c_2(X)\cdot H^{n-2}> 0$.
\end{lemma}
\begin{proof}
	The proof of this lemma is an adaptation of the proof of \cite[Proposition 4.8]{greb2016etale}. We start proving that if there exist ample line bundle $H_1,\dots , H_{n-2}$ on $X$ such that $c_2(X)\cdot H_1\cdots H_{n-2}=0$ then for any line bundle $L_1,\dots , L_{n-2}$ we have $c_2(X)\cdot L_1\cdots L_{n-2}=0$, \textsl{i.e.} $c_2(X)=0$ in $N^2(X)$. Since the ample cone is open in $N^1(X)$ and the intersection product is multilinear it is enough to prove that the intersection is trivial for $L_i$ ample line bundle.
	
	Up to taking large multiples of the divisors $H_i$ we can also assume that $H_i \pm L_i$ are ample divisors in $X$. We prove by induction on $k$ that $$c_2(X)\cdot (H_1+L_1)\cdots (H_k +L_k)\cdot H_{k+1}\cdots H_{n-2}= 0.$$
	For $k=0$ this is the hypothesis. Suppose it is true for $k$, we have
	$$0= c_2(X)\cdot (H_1+L_1)\cdots (H_k +L_k)\cdot H_{k+1}\cdots H_{n-2}=$$
	$$=c_2(X)\cdot (H_1+L_1)\cdots (H_k +L_k)\cdot (H_{k+1}\pm L_{k+1})\cdots H_{n-2}.$$
	Both the summands are non negative by Lemma \ref{miyaoka}, so they must be zero. For $k=n-2$ and expanding the product we get 
	$$0=\sum c_2(X)\cdot A_1 \cdots A_{n-2}$$ where $A_i\in \{H_i, L_i\}$. Since $A_i $ is nef, this is a zero sum of non-negative numbers whose sum is zero, so every summand must be zero. In particular we get $c_2(X)\cdot L_1 \cdots L_{n-2}=0$.
	
	To conclude we have to prove that if $c_2 (X) $ is non-zero then for any ample divisor $H$ the number $ c_2 (X) \cdot H^{n-2}$ is positive, but this is immediate since it is a non zero number by the above arguments, and it is non-negative by Lemma \ref{miyaoka}.
\end{proof}
\begin{remark}\label{chernnonzero}
	The second Chern class of a Calabi--Yau variety $X$ is non-zero. This is well-known under the further assumption that $X$ is smooth in codimension two. This is proved under the further assumption that $X$ is canonical and $\mathbb{Q}$-factorial in \cite[Theorem 1.4]{greb2016movable}. In a very recent paper it is proved that a normal projective variety with at most klt singularities, trivial canonical bundle and trivial second Chern class is a quasi-\'etale quotient of an abelian variety \cite[Theorem 1.2 and Remark 1.5]{lu2018acharacterization}, so the augmented irregularity is equal to the dimension. Note that the converse does not hold because $c_2 ((E\times E) / \pm )=24$. Using only the results in \cite{greb2016movable} we cannot prove that the second Chern class of a Calabi--Yau variety is non-zero because $X$ is not $\Q$-factorial, and a priori the second Chern class of a $\Q$-factorialization can be contracted in $X$.
\end{remark}

%% file: ampioeffettivo.tex
To prove Theorem \ref{newovercurves} we need some basic results.
The well-known statement for $\mathbb{Q}$-divisors that a nef divisor is big if and only if it has positive top self-intersection \cite[Theorem 2.2.13]{lazarsfeld2017positivity1} holds also for $\R$-divisors. This fact is well-known to the experts but we haven't found any references in the literature.

The following is an interesting consequence for variety with no rational curves and numerically trivial canonical bundle.
\begin{proposition}\label{bigamplecone}
	Let $X$ be a projective variety with at most log terminal singularities, numerically trivial canonical bundle and no rational curves. Then the ample cone and the big cone coincide.
\end{proposition}
\begin{proof}
	Let $D$ be any effective $\Q$-divisor. For small positive and rational $\varepsilon$ the pair $(X, \varepsilon D) $ is klt. Since there are no rational curves in $X$, the cone theorem \cite[Theorem 3.7]{kollar2008birational} tells us that $\varepsilon D$ is also nef. It follows that the effective cone is contained in the nef cone. Passing to the interior of such cones we get the thesis.
\end{proof}
Now following \cite{lazarsfeld2017positivity1} we define two cones that help us to study nef divisors that are not ample.
\begin{definition}
	The \emph{null cone} $\mathcal{N}_X\subset N^1(X)$ is the set of classes of divisors $D$ such that $D^n=0$. The \emph{boundary cone} $\mathcal{B}_X\subset N^1(X)$ is the boundary of the nef cone. 
\end{definition}
Note that these cones are not convex cones.
The following corollary that is already known by the experts explains the relation between these cones in our context.
\begin{proposition}
	Let $X$ be a variety with log terminal singularities, numerically trivial canonical bundle and no rational curves. The boundary of the ample cone is contained in the null cone, \textsl{i.e.} $\mathcal{B}_X\subset \mathcal{N}_X$.
\end{proposition}
\begin{proof}
	In the boundary of the ample cone there are nef $\mathbb{R}$-divisors that by Proposition \ref{bigamplecone} are not big $\mathbb{R}$-divisors. These $\mathbb{R}$-divisors has trivial top self-intersection and so they are in the null cone.
\end{proof}

%% file: cosasullaretta.tex
This proposition leads us to find (many) divisors with numerical dimension $n-1$ as explained in the following proposition.
\begin{proposition}\label{divisorebello}
	Let $X$ be a variety with log terminal singularities, numerically trivial canonical bundle and without rational curves. Let $H$ and $D$ be two divisors on $X$ that are respectively ample and nef of numerical dimension one. There is a (unique) rational number $t_0$  such that the $\mathbb{Q}$-divisor $\overline{N}(D,H)=H-t_0 \cdot D$ is nef and has numerical dimension $n-1$.
\end{proposition}
\begin{proof}
	The line in $N^1(X)$ for $t\in \mathbb{R}$ $$N_t=H+t\cdot D$$ gives us an interesting divisor in the intersection with the null cone.
	This line is parallel to the extremal ray of the nef cone generated by $[D]$. The divisor $D$ is nef so the line $N_t$ is contained in the nef cone for $t\geq \frac{-H^n}{nH^{n-1} \cdot D} $ and intersect the null cone when there is the equality. The divisor in the intersection $\overline{N}=H-\frac{H^n}{nH^{n-1}\cdot D} D$ 
	is a $\mathbb{Q}$-divisor because $H$ and $D$ are $\mathbb{Q}$-divisors and $\frac{H^n}{nH^{n-1}}\in \mathbb{Q}$. The divisor $\overline{N}$ has numerical dimension $n-1$ because $\overline{N}^{n-1}\cdot D=H^{n-1}\cdot D\neq 0 $ and it is not big. 
\end{proof}
In particular this proposition implies the following corollary.
\begin{corollary}\label{divisorec}
		Let $X$ be a variety with canonical singularities, numerically trivial canonical bundle and with no rational curves. If $c_2(X)\neq 0$ as element in $N^2(X)$ but $c_2(X)\cdot D=0$ in $N^3(X)$ for some nef $\mathbb{Q}$-divisor $D$ with $\nu(D)=1$, then there exists an ample $\mathbb{Q}$-divisor $H$ such that the $\mathbb{Q}$-divisor $\overline{N}(D,H)$ constructed in Proposition \ref{divisorebello} satisfies $c_2(X)\cdot \overline{N}(D,H)^{n-2}>0$.
\end{corollary}
\begin{proof}
	By Proposition \ref{divisorebello}, $X$ contains a $\mathbb{Q}$-divisor $N$ of numerical dimension $n-1$. By Lemma \ref{miyaoka} we know that the intersection of $c_2(X)$ with $n-2$ nef divisors is non negative. By Lemma \ref{intersecitonchern} for any ample divisor $H$ we have $H^{n-2}\cdot c_2(X)> 0$. By hypothesis $c_2(X)\cdot D= 0$, so $c_2(X)\cdot (\overline{N} (D,H))^{n-2}=c_2(X)\cdot (H-\frac{H^n}{nH^{n-1}\cdot D} D)^{n-2}= c_2(X) \cdot H^{n-2}> 0$. 
\end{proof}
%\begin{remark}
%	Note that this corollary is a generalization of Theorem 1.6 of \cite{div}. Conjecturally the two statements is the same, because in this context it is conjectured that any nef divisor is effective (abundance conjecture).
%	In fact in \cite{div} the authors start with a divisor $D $ that gives a fibration onto a curve whose general fibers are abelian varieties, that implies the hypotheses of our result.
%\end{remark}

%% file: dimostrazioneth.tex
Now the proof of Theorem \ref{newovercurves} follows from the results obtained in this section and by Theorem \ref{canonic}.

\begin{proof}[Proof of Theorem \ref{newovercurves}]
	If the singularities of $X$ are not canonical then $X$ is uniruled by \cite[Theorem 11]{kollar2009quotients}, so we can suppose that $X$ has canonical singularities. Moreover by Remark \ref{chernnonzero} we now that $c_2(X)\neq 0$.
	Suppose by contradiction that there are no rational curves in $X$. Thanks to Corollary \ref{divisorec} we can find a nef $\mathbb{Q}$-divisor $\overline{N}$ 
	such that $0<c_2(X) \cdot \overline{N}^{n-2}=12\operatorname{Td}_2(X)\cdot \overline{N}^{n-2}$. So applying \cite[Theorem 10]{kollar2015deformations} the divisor $\overline{N}$ induces an genus one fibration $X\rightarrow B$. Thus we can apply Theorem \ref{canonic} to find rational curves in $X$, which gives a contradiction. %If the augmented irregularity is non zero by \cite{greb2017klt} we have a decomposition of a quasi-\'etale cover $\tilde{X}$ of $X$ into the product of an abelian variety of dimension $\tilde{q}(X)$ and a variety $Z$ with $\tilde{q}(Z)=0$ and $K_Z\sim 0$. 
\end{proof}

%% file: conseguenzeultimoteorema.tex
The idea of Theorem \ref{newovercurves} is to find a nef divisor $D$ in $X$ with Itaka dimension $n-1$. In the proof of Theorem \ref{newovercurves} we explained that in our setting it is sufficient to find a nef $\mathbb{Q}$-divisor $D$ with numerical dimension $n-1 $ such that $D^{n-2}\cdot c_2(X)>0$. A careful analysis in dimension three can be found in \cite{diverio2011conjecture}. They work with smooth varieties but their proofs works verbatim also for Calabi--Yau varieties as in Definition \ref{calabiyau}.

%% file: Rational_Curves_on_Fibered_Varieties.bbl
\begin{thebibliography}{BHPVdV04}

\bibitem[Bea96]{beauville1996complex}
Arnaud Beauville.
\newblock {\em Complex algebraic surfaces}, volume~34 of {\em London
  Mathematical Society Student Texts}.
\newblock Cambridge University Press, Cambridge, second edition, 1996.
\newblock Translated from the 1978 French original by R. Barlow, with
  assistance from N. I. Shepherd-Barron and M. Reid.

\bibitem[BHPVdV04]{barth2015compact}
Wolf~P. Barth, Klaus Hulek, Chris A.~M. Peters, and Antonius Van~de Ven.
\newblock {\em Compact complex surfaces}, volume~4 of {\em Ergebnisse der
  Mathematik und ihrer Grenzgebiete. 3. Folge. A Series of Modern Surveys in
  Mathematics [Results in Mathematics and Related Areas. 3rd Series. A Series
  of Modern Surveys in Mathematics]}.
\newblock Springer-Verlag, Berlin, second edition, 2004.

\bibitem[BV04]{beauville2001chow}
Arnaud Beauville and Claire Voisin.
\newblock On the {C}how ring of a {$K3$} surface.
\newblock {\em J. Algebraic Geom.}, 13(3):417--426, 2004.

\bibitem[DF14]{diverio2011conjecture}
Simone Diverio and Andrea Ferretti.
\newblock On a conjecture of {O}guiso about rational curves on {C}alabi-{Y}au
  threefolds.
\newblock {\em Comment. Math. Helv.}, 89(1):157--172, 2014.

\bibitem[DFM16]{diverio2016rational}
Simone {Diverio}, Claudio {Fontanari}, and Diletta {Martinelli}.
\newblock {Rational curves on fibered Calabi-Yau manifolds}.
\newblock {\em arXiv e-prints}, page arXiv:1607.01561, Jul 2016.

\bibitem[Dru18]{druel2018decomposition}
St\'ephane Druel.
\newblock A decomposition theorem for singular spaces with trivial canonical
  class of dimension at most five.
\newblock {\em Invent. Math.}, 211(1):245--296, 2018.

\bibitem[FG65]{fischer1965lokal}
Wolfgang Fischer and Hans Grauert.
\newblock Lokal-triviale {F}amilien kompakter komplexer {M}annigfaltigkeiten.
\newblock {\em Nachr. Akad. Wiss. G\"ottingen Math.-Phys. Kl. II}, 1965:89--94,
  1965.

\bibitem[Ful84]{Ful84}
William Fulton.
\newblock {\em Intersection theory}, volume~2 of {\em Ergebnisse der Mathematik
  und ihrer Grenzgebiete (3)}.
\newblock Springer-Verlag, Berlin, 1984.

\bibitem[GGK17]{greb2017klt}
Daniel Greb, Henri Guenancia, and Stefan Kebekus.
\newblock Klt varieties with trivial canonical class-holonomy, differential
  forms, and fundamental groups.
\newblock {\em arXiv preprint arXiv:1704.01408}, 2017.

\bibitem[GKP16a]{greb2016etale}
Daniel Greb, Stefan Kebekus, and Thomas Peternell.
\newblock \'etale fundamental groups of {K}awamata log terminal spaces, flat
  sheaves, and quotients of abelian varieties.
\newblock {\em Duke Math. J.}, 165(10):1965--2004, 2016.

\bibitem[GKP16b]{greb2016movable}
Daniel Greb, Stefan Kebekus, and Thomas Peternell.
\newblock Movable curves and semistable sheaves.
\newblock {\em Int. Math. Res. Not. IMRN}, (2):536--570, 2016.

\bibitem[GKP16c]{greb2011singular}
Daniel Greb, Stefan Kebekus, and Thomas Peternell.
\newblock Singular spaces with trivial canonical class.
\newblock In {\em Minimal models and extremal rays ({K}yoto, 2011)}, volume~70
  of {\em Adv. Stud. Pure Math.}, pages 67--113. Math. Soc. Japan, [Tokyo],
  2016.

\bibitem[Gro67]{grothendieck1967elements}
A.~Grothendieck.
\newblock \'el\'ements de g\'eom\'etrie alg\'ebrique. {IV}. \'etude locale des
  sch\'emas et des morphismes de sch\'emas {IV}.
\newblock {\em Inst. Hautes \'Etudes Sci. Publ. Math.}, (32):361, 1967.

\bibitem[Har77]{hartshorne2013algebraic}
Robin Hartshorne.
\newblock {\em Algebraic geometry}.
\newblock Springer-Verlag, New York-Heidelberg, 1977.
\newblock Graduate Texts in Mathematics, No. 52.

\bibitem[Har10]{hartshorne2009deformation}
Robin Hartshorne.
\newblock {\em Deformation theory}, volume 257 of {\em Graduate Texts in
  Mathematics}.
\newblock Springer, New York, 2010.

\bibitem[HP18]{horing2018algebraic}
Andreas H{\"o}ring and Thomas Peternell.
\newblock Algebraic integrability of foliations with numerically trivial
  canonical bundle.
\newblock {\em Inventiones mathematicae}, pages 1--25, 2018.

\bibitem[Kaw91]{kawamata1991length}
Yujiro Kawamata.
\newblock On the length of an extremal rational curve.
\newblock {\em Invent. Math.}, 105(3):609--611, 1991.

\bibitem[KL09]{kollar2009quotients}
J\'anos Koll\'ar and Michael Larsen.
\newblock Quotients of {C}alabi-{Y}au varieties.
\newblock In {\em Algebra, arithmetic, and geometry: in honor of {Y}u. {I}.
  {M}anin. {V}ol. {II}}, volume 270 of {\em Progr. Math.}, pages 179--211.
  Birkh\"auser Boston, Inc., Boston, MA, 2009.

\bibitem[KM98]{kollar2008birational}
J\'anos Koll\'ar and Shigefumi Mori.
\newblock {\em Birational geometry of algebraic varieties}, volume 134 of {\em
  Cambridge Tracts in Mathematics}.
\newblock Cambridge University Press, Cambridge, 1998.
\newblock With the collaboration of C. H. Clemens and A. Corti, Translated from
  the 1998 Japanese original.

\bibitem[Kod63]{kodaira1963compact}
K.~Kodaira.
\newblock On compact analytic surfaces. {II}, {III}.
\newblock {\em Ann. of Math. (2) 77 (1963), 563--626; ibid.}, 78:1--40, 1963.

\bibitem[Kol15]{kollar2015deformations}
J.~Koll\'ar.
\newblock Deformations of elliptic {C}alabi-{Y}au manifolds.
\newblock In {\em Recent advances in algebraic geometry}, volume 417 of {\em
  London Math. Soc. Lecture Note Ser.}, pages 254--290. Cambridge Univ. Press,
  Cambridge, 2015.

\bibitem[Laz04]{lazarsfeld2017positivity1}
Robert Lazarsfeld.
\newblock {\em Positivity in algebraic geometry. {I}}, volume~48 of {\em
  Ergebnisse der Mathematik und ihrer Grenzgebiete. 3. Folge. A Series of
  Modern Surveys in Mathematics [Results in Mathematics and Related Areas. 3rd
  Series. A Series of Modern Surveys in Mathematics]}.
\newblock Springer-Verlag, Berlin, 2004.
\newblock Classical setting: line bundles and linear series.

\bibitem[LT18]{lu2018acharacterization}
Steven Lu and Behrouz Taji.
\newblock A characterization of finite quotients of abelian varieties.
\newblock {\em Int. Math. Res. Not. IMRN}, (1):292--319, 2018.

\bibitem[Miy87]{miyaoka1987chern}
Yoichi Miyaoka.
\newblock The {C}hern classes and {K}odaira dimension of a minimal variety.
\newblock In {\em Algebraic geometry, {S}endai, 1985}, volume~10 of {\em Adv.
  Stud. Pure Math.}, pages 449--476. North-Holland, Amsterdam, 1987.

\bibitem[MM83]{mori1983uniruledness}
Shigefumi Mori and Shigeru Mukai.
\newblock The uniruledness of the moduli space of curves of genus {$11$}.
\newblock In {\em Algebraic geometry ({T}okyo/{K}yoto, 1982)}, volume 1016 of
  {\em Lecture Notes in Math.}, pages 334--353. Springer, Berlin, 1983.

\bibitem[Ser06]{sernesi2007deformations}
Edoardo Sernesi.
\newblock {\em Deformations of algebraic schemes}, volume 334 of {\em
  Grundlehren der Mathematischen Wissenschaften [Fundamental Principles of
  Mathematical Sciences]}.
\newblock Springer-Verlag, Berlin, 2006.

\end{thebibliography}
